\documentclass[12pt]{article}
\usepackage{amsthm,amsfonts, amsbsy, amssymb,amsmath,graphicx}
\usepackage{graphics}
\usepackage[english]{babel}
\usepackage{graphicx}
\usepackage{lineno}
\usepackage{enumerate}
\usepackage{tikz}
\usepackage{tkz-graph}
\usepackage{hyperref}
\usepackage{color}
\usepackage{tikz}
\usepackage{algorithm}
\usepackage{enumitem}
\usetikzlibrary{arrows.meta,shapes.arrows}

\hypersetup{colorlinks=true, linkcolor=blue, citecolor=red,urlcolor=blue}

\newtheorem{theorem}{Theorem}
\newtheorem{lemma}[theorem]{Lemma}

\newtheorem{corollary}[theorem]{Corollary}

\newtheorem{remark}{Remark}

\newcommand{\edim}{\operatorname{edim}}

\title{Edge metric dimension of some graph operations}

\author{Iztok Peterin$^{(1,2)}$ and Ismael G. Yero$^{(3)}$\\
\\
$^{(1)}${\small University of Maribor,} {\small Faculty of Electrical Engineering and Computer Science}\\
{\small Koro\v{s}ka cesta 46, 2000 Maribor, Slovenia.} \\
$^{(2)}${\small Institute of Mathematics, Physics and Mechanics}\\
{\small Jadranska ulica 19, 1000 Ljubljana, Slovenia.} \\
{\small\it iztok.peterin\@@um.si} \\
$^{(3)}${\small Universidad de C\'adiz, Departamento de Matem\'aticas}\\
{\small} {\small Escuela Polit\'ecnica Superior de Algeciras, Av. Ram\'on Puyol s/n, 11202 Algeciras, Spain.} \\
{\small\it ismael.gonzalez\@@uca.es}
}

\date{}

\begin{document}

\maketitle

{\ }

\begin{abstract}
\noindent
Let $G=(V, E)$ be a connected graph. Given a vertex $v\in V$ and an edge $e=uw\in E$, the distance between $v$ and $e$ is defined as $d_G(e,v)=\min\{d_G(u,v),d_G(w,v)\}$. A nonempty set $S\subset V$ is an edge metric generator for $G$ if for any two edges $e_1,e_2\in E$ there is a vertex $w\in S$ such that $d_G(w,e_1)\ne d_G(w,e_2)$. The minimum cardinality of any edge metric generator for a graph $G$ is the edge metric dimension of $G$. The edge metric dimension of the join, lexicographic and corona product of graphs is studied in this article.

\medskip
\noindent{\it Keywords:} edge metric dimension; join of graphs; lexicographic product graphs; corona graphs

\medskip
\noindent{\it AMS Subject Classification Numbers:} 05C12; 05C76
\end{abstract}

\hrule

\section{Introduction and preliminaries}

Nowadays several parameters related to distances in graphs are highly attracting the attention of several researchers. One of them, namely, the metric dimension, has specifically centered several investigations. To see the richness of this topic, among several possible references, we would suggest for instance the three Ph. D. dissertations \cite{ale}, \cite{dk} and \cite{yunior}, and references cited therein. In this concern, a vertex $v$ of a connected graph $G$ \emph{distinguishes} (\emph{determines} or \emph{recognizes}) two vertices $u,w$ if $d(u,v)\ne d(w,v)$, where $d(x,y)$ represents the length of a shortest $x-y$ path in $G$. A subset of vertices $S$ of $G$ is a \emph{metric generator} for $G$, if any pair of vertices of $G$ is distinguished by at least one vertex of $S$. A metric generator of minimum cardinality is called a \emph{metric basis} for $G$ and its cardinality is the \emph{metric dimension} of $G$, which is denoted by $dim(G)$. These concepts were introduced by Slater in \cite{leaves-trees} in connection with some location problems in graphs. On the other hand, the concept of metric dimension was independently introduced by Harary and Melter in \cite{harary}.

A standard metric generator, as defined above, uniquely recognizes all the vertices of a graph in order to look out how they do ``behave'' in the graph. However, this ``surveillance'' does not succeed if an anomalous situations occurs in some edge between two vertices instead of in a vertex. It is possible that a metric generators properly identifies the edges in order to also see their behaving, but in general this is not possible. In this sense, in the way of correctly recognize the edges of a graph, a new parameter was recently introduced in \cite{KeTrYe}. Another variant on such direction was also presented in \cite{mix-dim} where not only edges are recognized between them, but also there is a recognition scheme between any two elements (vertices or edges) of a graph. In this work we only center our attention into recognizing edges.

Given a connected graph $G=(V,E)$, a vertex $v\in V$ and an edge $e=uw\in E$, the distance between the vertex $v$ and the edge $e$ is defined as $d_G(e,v)=\min\{d_G(u,v),d_G(w,v)\}$. A vertex $w\in V$ \emph{distinguishes} two edges $e_1,e_2\in E$ if $d_G(w,e_1)\ne d_G(w,e_2)$. A nonempty set $S\subset V$ is an \emph{edge metric generator} for $G$ if any two edges of $G$ are distinguished by some vertex of $S$. An edge metric generator with the smallest possible cardinality is called an \emph{edge metric basis} for $G$, and its cardinality is the \emph{edge metric dimension}, which is denoted by $\edim(G)$. In \cite{KeTrYe}, the concepts above were defined only for the case of connected graphs. However, if we consider non-connected graphs, then the parameter could be easily adapted by considering the distance between two vertices belonging to two different components as infinite. Nevertheless, such adapting make not much sense, since then we can readily seen the following.

\begin{remark}
If $G$ is a non-connected graph with components $G_1,\dots, G_r$, then
$$\edim(G)=\sum_{i=1}^r \edim(G_i).$$
\end{remark}

In connection with this remark above, we can consider every component of a graph separately. Moreover, any necessary concept, terminology and notation required in the work will be introduced throughout the exposition, right before it is firstly used, and for any remaining basic graph theory terminology we follow the book \cite{west}.

Studies of graph products have been intensively made in the last few decades and by now, a rich theory involving the structure and recognition of classes of these
graphs has emerged, cf. \cite{ImKl}. The most interesting and studied graph products are the Cartesian product, the strong product, the direct product and the
lexicographic product which are also called \emph{standard products}. There are several other operations made with graphs (from which some of them are also called as product graphs in the literature) that have also attracted the attention of several researchers. Some of them are for example the corona product, the join graphs, the rooted product and the hierarchical product. There are different styles of studying the graph products (or graph operations). One of them involves the analysis of the properties of the structure itself and a second one standard approach to graph products is to deduce properties of a product with respect to (usually the same) properties of its factors. This latter situation is the center of our work, in connection with the edge metric dimension. Some primary studies on the edge metric dimension of Cartesian product graphs were presented in \cite{KeTrYe}, where the value of the edge metric dimension was computed for the grid graphs (Cartesian product of two paths), and for some cases of torus graphs (Cartesian product of cycles). Moreover, some other results on this topic can be found in \cite{zubri}, where the edge metric dimension of the join graph $G\vee K_1$\footnote{The graph $G\vee K_1$ is obtained from a graph $G$ and a vertex $v$, by joining with and edge every vertex of $G$ with the vertex $v$.}, and of the Cartesian product of a path with any graph $G$ was studied. To the best of our knowledge there are no more results concerning the edge metric dimension of product graphs. In contrast with this fact, other variants of the standard metric dimension have been deeply studied in the last recent years. Some examples are for instance \cite{Barragan,Estrada,Jannesari} to just name those ones concerning the lexicographic product of graphs, which is one of the studied product of our work. In this sense, it is now our goal to make several contribution to this topic of edge metric dimension, and we precisely begin with studying the lexicographic product, the join and the corona product graphs.



\section{The join of graphs}

Given two graphs $G$ and $H$, the join graph $G\vee H$ is obtained from $G$ and $H$ by adding an edge between any vertex of $G$ and any vertex of $H$. In this sense, it is clearly observed that the join graph $G\vee H$ is always a connected independently of the connectivity of the graphs $G$ and $H$. We next study the edge metric dimension of join graphs.

To this end, we need the following terminology and notation. A set of vertices $D$ of a graph $G$ is a \emph{total dominating set} of $G$ if every vertex of $G$ is adjacent to a vertex of $D$. The minimum cardinality of any total dominating set of $G$ is the \emph{total domination number} of $G$ and is denoted by $\gamma_t(G)$. A set of vertices of cardinality $\gamma_t(G)$ is called a $\gamma_t(G)$-\emph{set}.

A graph $G$ is in a class of graphs $\mathcal G$ if for any vertex $u\in V(G)$ there is an edge $uv$ incident with $u$ ($v$ is a neighbor of $u$), such that $\{u,v\}$ is a $\gamma_t(G)$-set. It is easy to see that complete and complete bipartite graphs on at least two vertices are in $\mathcal G$. Let us mention that if $\gamma_t(G)$ does not exists, then $G$ contains an isolated vertex and is not in class $\mathcal G$, in particular $K_1$ is such.

\begin{theorem}
\label{join} For any non trivial graphs $G$ and $H$,
\begin{equation*}
\edim(G\vee H)=\left\{
\begin{array}{ll}
|V(G)|+|V(H)|-1, & \mbox{if $G\in\mathcal G$ or $H\in\mathcal G$}, \\[0.2cm]
|V(G)|+|V(H)|-2, & \mbox{if $G,H\notin\mathcal G$}. \\ %
\end{array}%
\right.
\end{equation*}
\end{theorem}

\begin{proof}
Let $G$ and $H$ be any graphs. The trivial upper bound for any graph is $\edim(G)\leq |V(G)|-1$ and therefore $\edim(G\vee H)\leq |V(G)|+|V(H)|-1$. Let $M$ be an edge metric basis for $G\vee H$. The distance between any edge $e=gh$ from $G\vee H$, where $g\in V(G), h\in V(H)$, to any vertex different from $g$ and $h$ equals one. Let $g$ be a fixed vertex from $G$. For $h\neq h'$, the edges $gh$ and $gh'$ have distance one to all vertices different from $g,h,h'$. Moreover, both edges have $g$ in common, so that the distance between them and $g$ equals zero in both cases. Therefore, at least one of $h$ and $h'$ must be in $M$. This argument can be repeated for any pair of edges of the form $gh$ and $gh'$ for any $h'\in V(H)-\{h\}$, and we see that at least $|V(H)|-1$ vertices from $H$ must be in $M$. Symmetrically we can see that at least $|V(G)|-1$ vertices from $G$ must be in $M$ and so, the lower bound $\edim(G\vee H)\geq |V(G)|+|V(H)|-2$ follows.

Assume now that one of the graphs, say $G$, belongs to $\mathcal G$. Suppose, with a purpose of contradiction, that $\edim(G\vee H)=|V(G)|+|V(H)|-2$ and let $S$ be an edge metric basis for $G\vee H$. In concordance with the above, it follows $|S\cap V(G)|=|V(G)|-1$ and $|S\cap V(H)|=|V(H)|-1$. Let $g\in V(G)$ and $h\in  V(H)$ be outside of $S$. Since $G\in \mathcal G$, $G$ has no isolated vertices, and there exists $g'\in V(G)$ which is adjacent to $g$ and $\{g,g'\}$ is a $\gamma_t(G)$-set. This implies that $gg'$ is at distance one to every vertex from $V(G)-\{g,g'\}$. Moreover, the edge $gg'$ is also at distance one to every vertex from $H$. As already mentioned the edge $g'h$ has distance one to every vertex from $V(G\vee H)-\{g',h\}$. Thus, the only vertices that distinguish the edges $gg'$ and $g'h$ are $g$ and $h$, which are not in $S$, and this is a contradiction with $S$ being an edge metric basis for $G\vee H$. Therefore, $\edim(G\vee H)>|V(G)|+|V(H)|-2$ and we have the equality $\edim(G\vee H)= |V(G)|+|V(H)|-1$.

To finish the proof let now $G,H\notin\mathcal G$. In this sense, there exists a vertex $g'\in V(G)$ such that $\{g,g'\}$ is not a $\gamma_t(G)$-set for every vertex $g$ adjacent to $g'$. In other words, there exists a vertex $x_g\in V(G)$ such that $d_G(gg',x_g)\geq 2$. Similarly, there exists a vertex $h'\in V(H)$ such that for every neighbor $h$ of $h'$ there exists a vertex $y_h\in V(H)$ such that $d_H(hh',y_h)\geq 2$. We will show that the set $S=V(G\vee H)-\{g',h'\}$ is an edge metric generator for $G\vee H$. Clearly, any two edges with both end-vertices in $S$ are distinguished by at least one vertex from $S$ (one of the end-vertices will do so). Similarly, edges $g'u$ and $g'v$, $u\neq v$, are also identified, because at least one of $u$ and $v$ is in $S$. The same happens with two edges $h'w$ and $h'z$, $w\neq z$. Thus, let $g'u$ and $h'v$ be two edges. If $u\neq v$, then without loss of generality, $u\in S$ distinguish $g'u$ and $h'v$. Hence, let $u=v$ and, by the symmetry of $G$ and $H$, we may assume that $u=g\in V(G)$. Because $G\notin\mathcal G$ for every edge, also for $gg'$, there exists a vertex $x_g\in V(G)$, at distance at least two to $gg'$ in $G$. Clearly, $x_g\in S$, $d_{G\vee H}(gh',x_g)=1$, $d_{G\vee H}(gg',x_g)=2$ and so, $x_g$ distinguishes the edges $gg'$ and $gh'$. Therefore, $S$ is an edge metric generator for $G\vee H$, and the equality $\edim(G\vee H)=|V(G)|+|V(H)|-2$ follows.
\end{proof}

It is easy to see that Theorem \ref{join} covers the results from \cite{KeTrYe} for a wheel $K_1\vee C_n$, for a fan $K_1\vee P_n$ and for a complete bipartite graph $K_{p,q}=N_p\vee N_q$. This latter result concerning bipartite graphs can be generalized for complete multipartite graphs as next shown.

\begin{corollary}\label{completebipartite}
For any complete multipartite graph $K_{r_1,\ldots,r_t}$,
\begin{equation*}
\edim(K_{r_1,\ldots,r_t})=\left\{
\begin{array}{ll}
r_1+r_2-2, & \mbox{if $t=2$,} \\
\sum_{i=1}^tr_i-1, & \mbox{if $t>2$.} \\ %
\end{array}%
\right.
\end{equation*}
\end{corollary}

\begin{proof}
For $t=2$ we have $\edim(K_{r_1,r_2})=\edim(N_{r_1}\vee N_{r_2})=r_1+r_2-2$ by Theorem \ref{join} because both $N_{r_1},N_{r_2}\notin\mathcal G$. For $t>2$ we can consider $K_{r_1,\ldots,r_t}$ as the join of $N_{r_1}\vee K_{r_2,\ldots,r_t}$. Since $K_{r_2,\ldots,r_t}\in\mathcal G$, by Theorem \ref{join}, we have that $\edim(K_{r_1,\ldots,r_t})=\sum_{i=1}^tr_i-1$.
\end{proof}


\section{The lexicographic product}

For two graphs $G$ and $H$ the vertex set of the \emph{lexicographic product graph} $G[H]$ is $V(G)\times V(H)$. Two vertices $(g,h)$ and $(g',h')
$ are adjacent in the lexicographic product $G[H]$ if either $gg'\in
E(G)$ or ($g=g'$ and $hh'\in E(H)$). For $h\in V(H)$, we call the set $G^{h}=\{(g,h)\in V(G[H]):g\in V(G)\}$, a $G$ \emph{layer} throughout $h$ in
$G[H]$, and $^{g}H=\{(g,h)\in V(G[H]):h\in V(H)\}$ is an $H$ \emph{layer} throughout $g$ in $G[H]$. Note that the subgraph of $G[H]$ induced by $G^{h}$ is isomorphic to $G$, as well as the subgraph of $G[H]$ induced by $^{g}H$ is isomorphic to $H$. Note also that the lexicographic product is associative but not commutative, cf.~\cite{ImKl}. The map $p_{G}:V(G[H])\rightarrow V(G)$ defined with $p_{G}((g,h))=g$ is called a \emph{projection map onto} $G$. Similarly, we can define the \emph{projection map onto} $H$. We must remark also that the lexicographic product graph $G[H]$ is a connected if and only if $G$ is connected. It is well known that the distance between any two vertices $(g,h),(g',h')\in V(G[H])$ is given by
\begin{equation}\label{dist-lexico}
  d_{G\circ H}((g,h),(g',h'))=\left\{\begin{array}{ll}
                                       \min\{2,d_H(h,h')\}, & \mbox{if $g=g'$}, \\
                                       d_G(g,g'), & \mbox{if $g\ne g'$}.
                                     \end{array}
  \right.
\end{equation}

Before we state the general result for the edge metric dimension of lexicographic product graphs we need to recall some well known concepts. Let $G$ be a graph. The vertex $u\in V(G)$ is called a \textit{false twin} of $v\in V(G)$ if $N_G(u)=N_G(v)$. Similarly, the vertices $x\in V(G)$ and $y\in V(G)$ are called \emph{true twins} if $N_G[x]=N_G[y]$. Clearly, each vertex is its own true twin and also its own false twin. Otherwise, different false twins $u$ and $v$ are not adjacent and if $x$ and $y$ are different true twins, then $x$ and $y$ are adjacent. It is easy to see that different false twins $u$ and $v$ are true twins only to themselves and vice versa, different true twins $x$ and $y$ are false twins only to themselves.

Both relations, being true or being false twins, are clearly equivalence relations. We are interested in all nontrivial equivalence classes of both relations, that is equivalence classes with at least two elements. Let $F=\{F_1,\ldots,F_k\}$ and $T=\{T_1,\ldots,T_{\ell}\}$ be the sets of all nontrivial equivalence classes of the false twin and true twin equivalence relations, respectively. Further, we need the number of elements in these classes, and therefore, we use notation $f(G)=\sum_{i=1}^k|F_i|$, $f'(G)=f(G)-k$, $t(G)=\sum_{i=1}^{\ell}|T_i|$ and $t'(G)=t(G)-\ell$.

The edges $e=uv$ and $e'=xy$ of a graph $G$ are \emph{twin edges} if $N_G[u]\cup N_G[v]=N_G[x]\cup N_G[y]$. If the twin edges $e=uv$ and $e'=xy$ are not incident with each other, then one can observe that vertices $u,v,x,y$ induce a subgraph that contains a four-cycle. Otherwise, if they are incident, say that $v=y$, then the condition $N_G[v]\subseteq N_G[u]\cup N_G[x]$ plays an important role. In particular, if there exists a neighbor $v$ of the false twins $u$ and $x$, such that $N_G[v]\subseteq N_G[u]\cup N_G[x]$, then the edges $uv$ and $xv$ are twin edges. Similarly, if $N_G[v]\subseteq N_G[u]\cup N_G[x]$ hold for a neighbor $v$ of the true twins $u$ and $x$, then again the edges $uv$ and $xv$ represent two twin edges.

As false twins, true twins, and twin edges play an important role while studying the edge metric dimension of lexicographic product graphs, we need to be careful to not count twice some of the vertices involved in the process. Let $G'$ be a graph obtained from a graph $G$ where we delete all vertices but one in every equivalence class of $F$ and of $T$. We call this operation \emph{twin deletion}. Let $Q=\{Q_1,\ldots,Q_m\}$ be nontrivial equivalence classes of the edge-twin relation of $G'$. We denote by $q(G')=\sum_{i=1}^m|Q_i|$ and $q'(G')=q(G')-m$ the number of nontrivial equivalence classes of the edge-twin relation.

It is easy to see that $t(K_n)=n, t'(K_n)=n-1$ and $f(K_n)=f'(K_n)=q(K'_n)=q'(K'_n)=0$. Similarly, for $m,n\geq 2$ we have that $f(K_{m,n})=m+n, f'(K_{m,n})=m+n-2$ and $t(K_{m,n})=t'(K_{m,n})=q(K'_{m,n})=q'(K'_{m,n})=0$. On the other hand, a graph $G$ obtained from a six-cycle $u_1u_2u_3u_4u_5u_6$ with two additional edges $u_2u_6$ and $u_3u_5$ has neither true nor false twins and we have $G'=G$. It is easy to see that $u_2u_3$ and $u_5u_6$ are twin edges and we have $q(G')=2$ and $q'(G')=1$. Another example is a graph $H$ that is obtained from vertices in $Q\cup N \cup Q'$, where the vertices of $Q$ and $Q'$ induce cliques while the set $N$ is an independent set of vertices and the cardinality of $Q,Q'$ and $N$ is at least two. In addition, we add all possible edges between $Q\cup Q'$ and $N$. The sets $Q$ and $Q'$ form two nontrivial equivalence classes of the true twin relation in $H$ and $N$ is the only nontrivial equivalence class of the false twin relation. Clearly, $H'$ is a path on three vertices $QNQ'$ and edges $QN$ and $NQ'$ are twins in $H'$.

We start this part of our exposition with a technical lemma that shows the independence of the values $q(G')$ and $q'(G')$ with respect to the deletion of true or false twins.

\begin{lemma}\label{technical}
Let $G$ be a graph. If the graphs $G'$ and $G''$ are obtained from $G$ by a twin deletion, then $q(G')=q(G'')$ and  $q'(G')=q'(G'')$.
\end{lemma}

\begin{proof}
Let $uv$ and $xy$ be twin edges of a graph $G$. If none of the vertices $u,v,x,y$ is a true or a false twin, then $uv$ and $xy$ remain twin edges in $G'$ and in $G''$. If exactly one of $u,v,x,y$, say $u$, is a true or a false twin in $G$, then there exist $u'$ and $u''$ from the same equivalence class as $u$ which remain in $G'$ and in $G''$, respectively, after twin deletion process. Clearly, $u'v$ and $u''v$ are twin edges with $xy$ in $G'$ and $G''$, respectively. Similarly, if more than one vertex from $u,v,x,y$ are true or false twins, then we can always find their representatives in $G'$ and in $G''$ with the same properties as $u,v,x,y$. These representatives have the same properties in $G'$ and in $G''$ and the number of twin edges remains the same in both $G'$ and in $G''$. Therefore, $q(G')=q(G'')$ and $q'(G')=q'(G'')$ follows immediately.
\end{proof}


\begin{theorem}\label{lex1}
Let $G$ be any graph with at least three vertices in every component and let $H\ncong K_1$ be a graph. Then
$$\edim(G[H])\ge |V(G)|(|V(H)|-1)+f'(G)+t'(G)+q'(G').$$
Moreover, if $H\notin\mathcal G$, then
$$\edim(G[H])= |V(G)|(|V(H)|-1)+f'(G)+t'(G)+q'(G').$$
\end{theorem}

\begin{proof}
Let $G$ be a graph where every component contains at least three vertices and let $H$ be any graph on at least two vertices. Let $S$ be an edge metric basis for $G[H]$.
For an edge $gg'\in E(G)$ and vertices $h,h'\in V(H)$, $h\neq h'$, by $(\ref{dist-lexico})$, we can see that the edges $(g,h)(g',h)$ and $(g,h)(g',h')$ have the same distance to every vertex $(g_0,h_0)$ different from $(g',h)$ and $(g',h')$. Therefore, at least one vertex from $(g',h)$ and $(g',h')$ must be in $S$. Because $h,h'$ and $g'$ are arbitrarily taken, and $G$ contains no isolated vertices, we see that $S$ contains at least $|V(H)|-1$ vertices in each layer $^{g'}\!H$, for every $g'\in V(G)$.

Let now $g$ and $g'$ be different true or false twins of $G$ and let $g_0$ be any neighbor of $g$ and $g'$ different from $g$ and $g'$. Notice that $g_0$ exists because every component of $G$ contains at least three vertices. If both layers $^g\!H$ and $^{g'}\!H$ have exactly $|V(H)|-1$ vertices in $S$, then let $(g,h)$ and $(g',h')$ be such vertices outside of $S$. In such a case, the edges $(g,h)(g_0,h)$ and $(g',h')(g_0,h)$ have different distances only to the vertices $(g,h)$ and $(g',h')$, because of $(\ref{dist-lexico})$, and since $g$ and $g'$ are true or false twins. This is a contradiction with $S$ being an edge metric basis. Therefore at least one layer from $^g\!H$ and $^{g'}\!H$ must be entirely contained in $S$ for any pair of true or false twins. So, if $T_1,\ldots,T_{\ell}$ are nontrivial equivalence classes of the true twin relation, then at most one layer $^g\!H$, $g\in T_i$, is not entirely contained in $S$ for every $i\in\{1,\ldots,\ell\}$. Similarly, if $F_1,\ldots,F_k$ are the nontrivial equivalence classes of the false twin relation, then at most one layer $^g\!H$, $g\in F_i$, is not entirely contained in $S$ for every $i\in\{1,\ldots,k\}$. With this comments we arrive to the lower bound $\edim(G[H])\geq |V(G)|(|V(H)|-1)+f'(G)+t'(G)$.

We will now increase this bound in the case when $q'(G')>0$ ($G'$ is obtained from $G$ by the twin deletion previously described). We consider an edge metric generator $S$ of cardinality $|V(G)|(|V(H)|-1)+f'(G)+t'(G)$ as described in the previous paragraph. By Lemma \ref{technical}, for any nontrivial equivalence class of true or false twins, we can choose any vertex $g$ as the representative vertex for which $^g\!H$ is not entirely contained in $S$. Suppose now that $uv$ and $xy$ are twin edges of $G'$. By our choice of $G'$, there exist vertices $(u,h_1),(v,h_2),(x,h_3)$ and $(y,h_4)$ from $G[H]$ which are not in $S$. Note that the condition $N_{G'}[u]\cup N_{G'}[v]=N_{G'}[x]\cup N_{G'}[y]$ implies that $N_G[u]\cup N_G[v]=N_G[x]\cup N_G[y]$ holds as well because any true or false twin $z$ of $G$ is in $N_G[w]$ if and only if the representative of $z$ in $G'$ is in $N_{G'}[w]$. But then no vertex from $S$ distinguishes the edges $(u,h_1)(v,h_2)$ and $(x,h_3)(y,h_4)$, which is a contradiction. Therefore, at least one layer from $^u\!H,^v\!H,^x\!H$ and $^{y}\!H$ must be entirely contained in $S$ for any pair of twins $uv$ and $xy$ from $G'$. Moreover, if $uv$ and $xy$ are incident, say that $v=y$, then at least one layer from $^u\!H$ and $^x\!H$ must be entirely contained in $S$. In consequence, if $Q_1,\ldots,Q_m$ are nontrivial equivalence classes of the edge-twin relation of $G'$, then for at most one edge $wz\in Q_i$ both $^w\!H$ and $^z\!H$ are not entirely contained in $S$ for every $i\in\{1,\ldots,m\}$. From this facts, the lower bound $\edim(G[H])\geq |V(G)|(|V(H)|-1)+f'(G)+t'(G)+q'(G')$ follows.

We next show that this lower bound is exact when $H\notin\mathcal G$. So, suppose that $H\notin\mathcal G$ and let $h'$ be a vertex from $H$ such that for every edge $hh'$ there exists a vertex $x_h\in V(H)$ where $x_h$ is neither adjacent to $h$ nor to $h'$. Notice that $h'$ can also be an isolated vertex of $H$, but then we can take for $h$ any vertex different from $h'$ which exists by the assumption. For every nontrivial equivalence class $T_i$ of the true twin relation, we fix one vertex $t_i\in T_i$, $i\in\{1,\ldots,\ell\}$, and for every nontrivial equivalence class $F_j$ of the false twin relation, we also fix one vertex $f_i\in F_i$, $i\in\{1,\ldots,k\}$. Finally, let $Q_1,\ldots,Q_m$ be the nontrivial equivalence classes of the edge-twin relation of $G'$. We fix one edge $w_iz_i\in Q_i$ for every $i\in\{1,\ldots,m\}$ and for every other edge $u^j_iv^j_i\in Q_i$, $j\in\{1,\ldots,|Q_i|-1\}$, we fix a vertex which must be different from $w_i$ and from $z_i$. If the notation is chosen so that $u_i^j$ is this vertex, then we denote by $Q'_i$ the set $\{u_i^1,\ldots,u_i^{|Q_i|-1}\}$. We define the new sets $\mathcal T=\bigcup_{i=1}^{\ell}(T_i-\{t_i\})$, $\mathcal F=\bigcup_{i=1}^k(F_i-\{f_i\})$ and $\mathcal Q=\bigcup_{i=1}^m Q'_i$. We will show that the set
$$S=\left(V(G)\times (V(H)-\{h'\})\right)\cup\left((\mathcal T\cup\mathcal F\cup\mathcal Q)\times\{h'\}\right)$$
is an edge metric generator of cardinality $|V(G)|(|V(H)|-1)+f'(G)+t'(G)+q'(G')$. The equality $|S|=|V(G)|(|V(H)|-1)+f'(G)+t'(G)+q'(G')$  follows directly from the definitions of $\mathcal T, \mathcal F$ and $\mathcal Q$. To observe that $S$ is an edge metric generator, we need to check only pairs of edges $e_1$ and $e_2$ that have both end-vertices outside of $S$, or if $e_1$ and $e_2$ are incident, then the common end-vertex can be in $S$. If $e_1$ and $e_2$ (incident or not) have both end-vertices outside of $S$, then they must be lying over the layer $^{h'}\!H$. Let $e_1=(u,h')(v,h')$ and $e_2=(x,h')(y,h')$. By the definition of $S$ we have that $N_G[u]\cup N_G[v]\neq N_G[x]\cup N_G[y]$. Suppose without loss of generality that there exists $g\in (N_G[u]\cup N_G[v])-(N_G[x]\cup N_G[y])$. Clearly, $(g,h)\in S$ distinguishes $e_1$ and $e_2$ because  $d_{G[H]}((g,h),e_1)=1$ and $d_{G[H]}((g,h),e_2)>1$ (recall that $h$ is different from $h'$ as they are adjacent or $h'$ is an isolated vertex). It remains that $e_1$ and $e_2$ are incident and that the common vertex is from $S$. Let now $e_1=(u,h')(v,h'')$ and $e_2=(x,h')(v,h'')$. If $u\neq v$ and $x\neq v$, then we can use the same argument that $N_G[u]\cup N_G[v]\neq N_G[x]\cup N_G[y]$ and we conclude as before. Thus we may assume that either $u=v$ or $x=v$ and in this case $h''$ must be adjacent to $h'$. By the symmetry we can assume that $u=v$. Recall that, since $H\notin\mathcal G$, there exists a vertex $x_{h''}$ that is nonadjacent to $h'$ and nonadjacent to $h''$. Clearly, $(u,x_{h''})\in S$ distinguishes $e_1$ and $e_2$ because $d_{G[H]}((u,x_{h''}),e_1)=2$ and $d_{G[H]}((u,x_{h''}),e_2)=1$. Because every pair of edges from $G[H]$ is distinguished by a vertex from $S$, we obtain that $\edim(G[H])\leq |V(G)|(|V(H)|-1)+f'(G)+t'(G)+q'(G')$ when $H\notin\mathcal G$ and the equality follows for this case.
\end{proof}

One could think that the bound given above is indeed an equality for any graphs $G$ and $H$ satisfying the statements of the theorem. However, this is not true, since other extra situations are also influencing the value of $\edim(G[H])$. We next comment some facts on this regard. To this end, we need the following terminology.

A vertex $v\in V(G)$ is called a \emph{satellite} of a vertex $u\in V(G)$ if $N_G[v]\subsetneq N_G[u]$. We also say that in such a case $u$ has a satellite $v$. If $v$ is a satellite of $u$, then $v$ and $u$ are adjacent. If $u$ has a false twin $v$, $u\neq v$, then $u$ cannot have a satellite, as any vertex $w$ adjacent to $u$ has $v$ in its closed neighborhood. On the other hand $u$, can be a satellite if it has a false twin $v$. If $x$ has a different true twin $y$, then $x$ can be a satellite of some vertex and can also have satellites. Similar as true twins, false twins and twin edges, vertices that have satellites are important for the edge metric dimension of the lexicographic product $G[H]$ when $H\in \mathcal G$.

\begin{lemma}\label{satellite}
Let $G$ and $H$ be any graphs, where every component of $G$ contains at least three vertices and $H\in\mathcal G$, and let $S$ be an edge metric basis for $G[H]$. If $g\in V(G)$ has a satellite $g'\in V(G)$, then $^{g}\!H$ or $^{g'}\!H$ is entirely contained in $S$.
\end{lemma}

\begin{proof}
By the proof of Theorem \ref{lex1}, there are at least $|V(H)|-1$ vertices of $^{g}\!H$ and of $^{g'}\!H$ in $S$. Suppose that, on the way to a contradiction, $(g,h)$ and $(g',h')$ do not belong to $S$. Because $H\in\mathcal G$, there exists a vertex $h_0\in V(H)$ that is adjacent to $h$ and $\{h,h_0\}$ is a $\gamma_t(H)$-set. We will see that edges $e_1=(g,h)(g,h_0)$ and $e_2=(g',h')(g,h_0)$ are not distinguished by any vertex form $S$. First, every vertex from $^{g}\!H$, with the exception of $(g,h)$ and $(g,h_0)$, is at distance 1 to both $e_1$ and $e_2$, because $\{h,h_0\}$ is a $\gamma_t(H)$-set and $gg'\in E(G)$. Second, the vertices from $^{g_0}\!H$ for every $g_0\in N_G[g']$, with the exception of $(g',h')$, are at distance 1 to $e_1$ and to $e_2$, because $g'$ is a satellite of $g$. Third, and finally, the other vertices from $G[H]$ are at the same distance to vertices of $^{g}\!H$, and so, also to $e_1$ and to $e_2$. Thus, the only vertices that distinguish $e_1$ and $e_2$ are $(g,h)$ and $(g',h')$ which are not in $S$, and this is a contradiction. Therefore, $^{g}\!H$ or $^{g'}\!H$ is entirely contained in $S$.
\end{proof}

This last lemma is one of the reasons causing that the bound from Theorem \ref{lex1} does not in general hold as equality, when $H\in\mathcal G$. We can expect that, if there are some satellite vertices, then one needs to add some additional vertices to a given set to become an edge metric generator. Again we observe that we need to be careful not to count twice some of them. For example, observe $K_p\vee N_r$. Every vertex from $N_r$ is a satellite from every vertex from $K_p$ and one can expect that $\min\{p,r\}$ of vertices need to be added to a given set to get an edge metric generator for some $H\in\mathcal G$, which yields a kind of minimization problem. However, this is not the right approach, as we have already seen in the second paragraph of the proof of Theorem \ref{lex1}, where all $H$-layers initiated by true twins (with one possible exception in every equivalence class), and by false twins (with one possible exception in every equivalence class) must belong to a given edge metric basis.


\section{The corona product}

Let $G$ and $H$ be two graphs of order $n_1$ and $n_2$, respectively. The corona product graph $G\odot H$ is defined as the graph obtained from $G$ and $H$, by taking one copy of $G$ and $n_1$ copies of $H$ and joining by an edge every vertex from the $i^{th}$-copy of $H$ with the $i^{th}$-vertex of $G$. Given a vertex $g\in V(G)$, the copy of $H$ whose vertices are adjacent to $g$ is denoted by $H_g$. We will first analyze the situation in which the second factor of this product is not isomorphic to the singleton graph $K_1$.

\begin{theorem}
\label{corona} For any graphs $G$ and $H$ where $G$ is connected and $|V(H)|\geq 2$,
$$\edim(G\odot H)=|V(G)|\cdot(|V(H)|-1)$$
\end{theorem}

\begin{proof}
Let $G$ and $H$ be any graphs and let $n=|V(H)|\geq 2$. Let $g\in V(G)$ and let $\{g_1,\ldots,g_n\}$ be the set of vertices of the copy $H_g$ of $H$. Any two edges $gg_i$ and $gg_j$, $i\neq j$, have the same distance to all vertices form $V(G\odot H)-\{g_i,g_j\}$. Therefore, at least one of them must be in any  edge metric basis of $G\odot H$. Because $i,j$ and $g$, with $1\leq i<j\leq n$, are arbitrary, we see that every edge metric basis of $G\odot H$ must contain at least $n-1$ vertices from every copy $H_g$ of $H$ in $G\odot H$. This yields the lower bound  $\edim(G\odot H)\geq |V(G)|\cdot(|V(H)|-1)$.

On the other hand, let $H_g$ be a copy of $H$ corresponding to a vertex $g\in V(G)$. We will show that the set $S=\cup_{g\in V(G)} (V(H_g)-\{h\})$ is an edge metric generator for $G\odot H$, where $h$ is an arbitrary vertex of $V(H_g)$. First notice that $S$ is nonempty because $n\geq 2$, and by the same reason in every copy of $H$ exists at least one vertex from $S$. Any two different edges from one copy of $H$ are distinguished by some vertex in $S$, because at least two end vertices of these two edges are in $S$. The same argument also holds if both edges are in two different copies of $H$. Similarly, an edge from $G$ and one from any copy of $H$ are distinguished by at least one end-vertex of the edge lying in the copy of $H$ which is in $S$. Two different edges from $G$ are distinguished by at least one vertex $g\in V(G)$ and so, any vertex in $S\cap V(H_g)$ distinguishes these two edges. So, let now consider one edge $gh'$ with $g\in V(G)$ and $h'\in V(H_g)$. If the second edge is $gh_1$ for some $h_1\in V(H_g)$ and $h_1\neq h'$, then these two edges are distinguish by $h'$ or by $h_1$. If the second edge $g'g''$ belongs to $G$, then at least one end-vertex, say $g'$, is different from $g$ and they are distinguished by any vertex from $S\cap V(H_{g'})$. Any edge $gh'$ is also clearly distinguished from any edge with at least one end-vertex in other copy of $H$ by any vertex from $S$ in that copy. So, let finally the second edge $h_1h_2$ be from $H_g$. Any vertex $x\in S\cap H_{g'}$, where $g'\ne g$, distinguishes $gh'$ and $h_1h_2$ because $d_{G\odot H}(h_1h_2,x)=d_{G\odot H}(hg,x)+1$. Therefore, we have $\edim(G\odot H)\le |V(G)|\cdot(|V(H)|-1)$ whenever $n\geq 2$, and the equality follows.
\end{proof}

In contrast with the case above, the corona product graph $G\odot K_1$ is in general complicated to deal with. In order to observe this, the following terminology and notation will be required. A vertex of degree at least $3$ in a tree $T$ will be called a \emph{major vertex} of $T$.
Any leaf $u$ of $T$ is said to be a \emph{terminal vertex} of a major vertex $v$ of $T$ if
$d(u, v)<d(u,w)$ for every other major vertex $w$ of $T$. The \emph{terminal degree}  of
a major vertex $v$ is the number of terminal vertices of $v$. A major vertex $v$ of $T$ is an
\emph{exterior major vertex} of $T$ if it has positive terminal degree. Let $n_1(T)$ denote the number of leaves of $T$, and let $ex(T)$ denote the number of exterior major vertices of $T$. We can now state the formula for the edge metric dimension of a tree given in \cite{KeTrYe}. If $T$ is a tree that is not a path, then
\begin{equation}\label{formula-trees}\edim(T) = n_1(T) - ex(T).\end{equation}

Some situations can be easily deduced for $G\odot K_1$. For instance, if $G$ is the path $P_2$, then clearly $G\odot K_1\cong P_4$ and so, $\edim(P_2\odot K_1)=\edim(P_4)=1$. Also, if  $G$ is a path of order $n\ge 3$, then $G\odot K_1$ is a tree such that $n_1(G\odot K_1)=n$ and $ex(G\odot K_1)=n-2$. Thus, from (\ref{formula-trees}) we get $\edim(P_n\odot K_1)=2$.

\begin{theorem}
For any graph $G$, $\edim(G\odot K_1)\ge \edim(G)$, and this bound is sharp.
\end{theorem}

\begin{proof}
Let $V(G)=\{g_1,\dots,g_n\}$ and for every $g_i\in V(G)$, let $u_i$ be the vertex adjacent to $g_i$ corresponding to the copy of $K_1$ used in the corona product. Assume $S$ is an edge metric basis for $G\odot K_1$, and consider the set of vertices $S'=\{g_i\,:\,\{g_i,u_i\}\cap S\ne \emptyset\}$. We will prove that $S'$ is an edge metric generator for $G$. Let $e,f\in E(G)$ be any two edges and let $x\in S$ such that $d_G(e,x)\ne d_G(f,x)$. If $x\in V(G)$, then $x\in S'$ and so, $x$ determines $e$ and $f$. If $x\notin V(G)$, then $x$ is a vertex corresponding to a copy of $K_1$ and is adjacent to a vertex $x'\in V(G)$ which is also in $S'$. Thus,  $d_G(e,x')=d_{G\odot K_1}(e,x)-1\ne d_G(f,x)-1=d_{G\odot K_1}(f,x')$ and so, $x'$ distinguish $e$ and $f$. As a consequence, $S'$ is an edge metric generator for $G$ and the bound follows.

To see the sharpness of the bound we consider the graph $K_n\odot K_1$ where $n\ge 3$. Let $V(K_n)=\{g_1,\dots,g_n\}$ and, as above, for every $g_i\in V(K_n)$, let $u_i$ be the vertex adjacent to $g_i$ corresponding to the copy of $K_1$ used in the corona product. Assume $S$ is an edge metric basis for $K_n$ (note that an edge metric basis of $K_n$ is form by any set of $n-1$ vertices of $K_n$), and consider the set of vertices $S'=\{u_i\,:\,g_i\in S\}$. We shall prove $S'$ is a edge metric generator for $K_n\odot K_1$. Let $e,f\in E(K_n\odot K_1)$ be any two edges and consider the following situations.
\begin{itemize}[leftmargin=*]
  \item $e,f\in E(K_n)$. Since there exists a vertex $g_j\in S$ such that $d_{K_n}(e,g_j)\ne d_{K_n}(f,g_j)$, we deduce that
  $d_{K_n\odot K_1}(e,u_j)=d_{K_n}(e,g_j)+1\ne d_{K_n}(f,g_j)+1=d_{K_n\odot K_1}(f,u_j)$. As $u_j\in S'$, we have that $u_j$ recognizes $e,f$.

  \item $e\notin E(K_n)$ and $f\in E(K_n)$. Let $e=g_iu_i$ and $f=g_jg_k$. If $u_i\in S'$, then clearly $e,f$ are identified by $u_i$. If $u_i\notin S'$ and ($u_j\in S'$ or $u_k\in S'$), say $u_j\in S'$, then $e,f$ are distinguished by $u_j$. Since the case whether $u_i,u_j,u_k\notin S'$ is not possible, because there is only one vertex of $V(K_n)$ which is not in $S$, we are done with this case.

  \item $e,f\notin E(K_n)$. Let $e=g_iu_i$ and $f=g_ju_j$. Since there is only one vertex of $V(K_n)$ which is not in $S$, it must happen $u_i\in S'$ or $u_j\in S'$. Thus, $e,f$ are distinguished by $u_i$ or by $u_j$.
\end{itemize}
As a consequence of the cases above we obtain that $S'$ is an edge metric generator for $K_n\odot K_1$ and so, $\edim(K_n\odot K_1)\le \edim(K_n)$. Since $\edim(G\odot K_1)\ge \edim(G)$ for any graph $G$, we obtain the equality $\edim(K_n\odot K_1)=\edim(K_n)$ and the sharpness of the bound is completed.
\end{proof}

Although the bound above is tight, it is possible to observe that the difference between $\edim(G\odot K_1)$ and $\edim(G)$ can be arbitrarily large. To observe this, we consider a tree $T$ of order $n\ge 3$ which is not a path. Clearly, $n_1(T\odot K_1)=n$ and $ex(T\odot K_1)=n-n_1(T)$. As a consequence of (\ref{formula-trees}), $\edim(T\odot K_1)=n_1(T\odot K_1)-ex(T\odot K_1)=n-(n-n_1(T))=n_1(T)$ and so, $\edim(T\odot K_1)-\edim(T) = n_1(T)-(n_1(T)-ex(T))=ex(T)$, which can be as large as we would require.



\begin{thebibliography}{9}

\bibitem{Barragan} Barrag\'an-Ram\'irez, G. A., Estrada-Moreno, A., Ram\'irez-Cruz, Y., Rodr\'iguez-Vel\'azquez, J. A.: The local metric dimension of the lexicographic product of graphs, Bull. Malays. Math. Sci. Soc. To appear. (2018)

\bibitem{ale} Estrada Moreno, A.: On the $(k,t)$-metric dimension of a graph, Ph. D. dissertation. Universitat Rovira i Virgili, 2016

\bibitem{Estrada} Estrada-Moreno, A., Yero, I. G., Rodr\'iguez-Vel\'azquez, J. A.: The $k$-metric dimension of the lexicographic product of graphs, Discrete Math. 339(7) 1924-1934 (2016)

\bibitem{ImKl} Hammack, R., Imrich, W., Klav\v{z}ar, S.: Handbook of Product Graphs, Second Edition, CRC Press, Boca Raton, FL, 2011

\bibitem{harary} Harary, F., Melter, R. A.: On the metric dimension of a graph, Ars Combin. 2, 191--195 (1976)

\bibitem{Jannesari} Jannesari, M., Omoomi, B.: The metric dimension of the lexicographic product of graphs, Discrete Math. 312(22), 3349-3356 (2012)

\bibitem{mix-dim} Kelenc, A., Kuziak, D., Taranenko, A., Yero, I. G.: On the mixed metric dimension of graphs, Appl. Math. Comput. 314, 429--438 (2017)

\bibitem{KeTrYe} Kelenc, A., Tratnik, N., Yero, I. G.: Uniquely identifying the edges of a graph: the edge metric dimension, Discrete Appl. Math. in press (2018)

\bibitem{dk} Kuziak, D.: Strong resolvability in product graphs, Ph. D. dissertation, Universitat Rovira i Virgili, 2014

\bibitem{yunior} Ram\'irez Cruz, Y.: The simultaneous (strong) metric dimension of graph families﻿, Ph. D. dissertation, Universitat Rovira i Virgili, 2016

\bibitem{leaves-trees} Slater, P. J.: Leaves of trees, Congr. Numer. 14, 549--559 (1975)

\bibitem{west} West, D. B.: Introduction to Graph Theory (Second Edition), Prentice Hall, USA, 2001

\bibitem{zubri} Zubrilina, N.: On the edge dimension of a graph, Discrete Math. 341, 2083--2088 (2018)

\end{thebibliography}
\end{document}